\documentclass[11pt]{article}

\usepackage{amsmath} 
\usepackage[mathscr]{eucal}
\usepackage{amssymb} 
\usepackage{theorem}
\usepackage{tikz-cd}
\usetikzlibrary{positioning}

\theoremstyle{change}  
\newtheorem{theorem}{Theorem}[section] 
\newtheorem{proposition}[theorem]{Proposition}
\newtheorem{corollary}[theorem]{Corollary}
\theorembodyfont{\rmfamily}  
\newtheorem{Remark}[theorem]{Remark}

\newtheorem{definition}[theorem]{Definition}

\newtheorem{nothing}[theorem]{} 

\newenvironment{proof}{\noindent{\bf Proof}\ }{\qed\bigskip}

\renewcommand{\le}{\leqslant} 

\newcommand{\Bl}{\mathrm{Bl}}

\newcommand{\calO}{\mathcal{O}}

\newcommand{\KK}{\mathbb{K}}

\newcommand{\lexp}[2]{\setbox0=\hbox{$#2$} \setbox1=\vbox to
                 \ht0{}\,\box1^{#1}\!#2}

\newcommand{\qed}{\nobreak\hfill
                  \vbox{\hrule\hbox{\vrule\hbox to 5pt
                  {\vbox to 8pt{\vfil}\hfil}\vrule}\hrule}}

\newcommand{\TD}{T^\Delta}

\title{On $p$-permutation equivalences between direct products of blocks}
\author{Deniz Y\i lmaz}
\date{}
\providecommand{\keywords}[1]
{
  \small\smallskip\par	
  \hspace{2ex}\textbf{Keywords:} #1
}
\providecommand{\msc}[1]
{
  \small\smallskip\par	
  \hspace{2ex}\textbf{MSC2020:} #1%
}

\begin{document}

\sloppy

\maketitle

\begin{abstract}
We extend the notion of a $p$-permutation equivalence to an equivalence between direct products of block algebras. We prove that a $p$-permutation equivalence between direct products of blocks gives a bijection between the factors and induces a $p$-permutation equivalence between corresponding blocks.
\end{abstract}

\keywords{block, $p$-permutation equivalence, trivial source modules.}
\msc{20C20, 19A22.}


\section{Introduction}

One of the main themes in representation theory of finite groups is to study equivalences between block algebras. Various authors have defined different notions of equivalences, such as Puig equivalence \cite{Puig1999},  splendid Rickard equivalence \cite{Rickard1996}, derived equivalence, isotypy, perfect isometry \cite{Broue1990}, $p$-permutation equivalence \cite{BoltjeXu2008},\cite{BoltjePerepelitsky2020}, and functorial equivalence \cite{BoucYilmaz2022}. Our aim in this paper is to extend the notion of a $p$-permutation equivalence to an equivalence between direct products of blocks. 

Let $G$ and $H$ be finite groups. Let $p>0$ be a prime and let $(\KK,\calO,k)$ denote a $p$-modular system where $\calO$ is a complete discrete valuation ring with residue field $k$ of characteristic $p$ and field of fractions $\KK$ of characteristic $0$. Suppose that $\calO$ contains a root of unity whose order is equal to the exponent of $G\times H$. 

Let $A$ be a sum of blocks of $\calO G$ and $B$ a sum of blocks of $\calO H$. Let $\TD(A,B)$ denote the Grothendieck group with respect to split short exact sequences of $p$-permutation $(A,B)$-bimodules whose indecomposable summands have twisted diagonal vertices when regarded as $\calO [G\times H]$-modules. In \cite{BoltjePerepelitsky2020}, Boltje and Perepelitsky define a $p$-permutation equivalence between $A$ and $B$ as an element $\gamma\in\TD(A,B)$ such that
\begin{align*}
\gamma\cdot_H \gamma^\circ =[A]\in \TD(A,A) \quad \text{and}\quad  \gamma^\circ\cdot_G\gamma =[B]\in \TD(B,B)
\end{align*}
where $\gamma^\circ$ is the $\calO$-dual of $\gamma$ and where $\cdot_H$ is tensor product over $\calO H$. Among many other interesting and important properties of $p$-permutation equivalences, they proved that if $\gamma$ is a $p$-permutation equivalence between $A$ and $B$, then there is a bijection between the block summands of $A$ and $B$ and $\gamma$ induces a $p$-permutation equivalence between the corresponding blocks, see \cite[Theorem~10.10]{BoltjePerepelitsky2020}. We show that a similar phenomenon holds for $p$-permutation equivalences between direct products of blocks.

\begin{definition}
Let $A_1\times\cdots\times A_n$, $B_1\times\cdots \times B_m$ and $C_1\times\cdots\times C_l$ be direct products of block algebras of finite groups. Let $\gamma=(\gamma_{ij})$ and $\gamma'=(\gamma'_{jk})$ be matrices with entries $\gamma_{ij}\in \TD(A_i,B_j)$ and $\gamma'_{jk}=\TD(B_j,C_k)$. We denote by $\gamma\circ\gamma'$ the product of the matrices $\gamma$ and $\gamma'$. More precisely,
\begin{align*}
(\gamma\circ\gamma')_{i,k}=\sum_{j=1}^m \gamma_{ij}\cdot_{H_j}\gamma'_{jk} \in \TD(A_i,C_k)\,.
\end{align*}
\end{definition}

\begin{definition}
Let $G_1,\cdots,G_n$ and $H_1,\cdots,H_m$ be finite groups. Let $A_i\in\Bl(\calO G_i)$ and $B_j\in\Bl(\calO H_j)$ be block algebras for $i=1,\cdots, n$ and $j=1,\cdots,m$. A {\em $p$-permutation equivalence} between the direct product algebras $A_1\times\cdots\times A_n$ and $B_1\times\cdots \times B_m$ is a matrix $\gamma=\left(\gamma_{ij}\right)$ where $\gamma_{ij}\in \TD(A_i,B_j)$ such that 
\begin{equation*}
\gamma\circ\gamma^\circ=\begin{pmatrix}
[A_1] & 0 & \cdots & 0 \\
0 & [A_2] & 0\cdots &0\\
\vdots & \vdots & \ddots & \vdots \\
0&\cdots& 0 \cdots & [A_n]
\end{pmatrix}
 \quad \text{and} \quad \gamma^\circ\circ\gamma=\begin{pmatrix}
[B_1] & 0 & \cdots & 0 \\
0 & [B_2] & 0\cdots &0\\
\vdots & \vdots & \ddots & \vdots \\
0&\cdots& 0 \cdots & [B_m]
\end{pmatrix}
\end{equation*}
where $\gamma^\circ = ((\gamma_{i,j}^\circ)_{ij})^t$. 
\end{definition}

Our main result is the following.

\begin{theorem}\label{thm main thm}
Let $G_1,\cdots,G_n$ and $H_1,\cdots,H_m$ be finite groups. Let $A_i\in\Bl(\calO G_i)$ and $B_j\in\Bl(\calO H_j)$ be block algebras for $i=1,\cdots, n$ and $j=1,\cdots,m$.  Assume that $\calO$ contains a root of unity of order the exponent of $G_i$ and $H_j$ for each $i$ and $j$.  Let $\gamma=(\gamma_{ij})$ be a $p$-permutation equivalence between the direct products $A_1\times\cdots\times A_n$ and $B_1\times\cdots\times B_m$ of block algebras.  Then $n=m$ and in each row and in each column of $\gamma$, there exists precisely one non-zero element. Moreover, if $\gamma_{ij}$ is the non-zero element in the $i$-th row and $j$-th coloumn, then $\gamma_{ij}$ is a $p$-permutation equivalence between $A_i$ and $B_j$. 
\end{theorem}

\section{The proof of the main theorem}

Throughout $G, G_1,\cdots, G_n$, $H,H_1,\cdots, H_m$ denote finite groups. Also, $(\KK,\calO,k)$ denotes a $p$-modular system where $\calO$ is a complete discrete valuation ring with residue field $k$ of characteristic $p$ and field of fractions $\KK$ of characteristic $0$. We suppose that $\calO$ contains a root of unity of order the exponent of $G,G_i,H$ and $H_j$ for all $i$ and $j$. We follow the proof of \cite[Theorem~10.10]{BoltjePerepelitsky2020} closely. 

\begin{nothing}
We denote by $R(\KK G)$ and $R(kG)$ the Grothendieck groups with respect to short exact sequences of $\KK G$-modules and $kG$-modules, respectively, and by $T(\calO G)$ and $T(kG)$ the Grothendieck groups with respect to split short exact sequences of $p$-permutation $\calO G$-modules and $p$-permutation $kG$-modules, respectively.

We denote by $-^*$ the anti-involution $g\mapsto g^{-1}$ of any group algebra of a group $G$. If $A$ is a block of $\calO G$ and $B$ is a block of $\calO H$, then we can regard any $(A,B)$-bimodule $M$ as an $A\otimes B^*$-module via the isomorphism $\calO (G\times H)\cong \calO G\otimes_\calO \calO H$. We set $R(\KK G,\KK H):=R(\KK [G\times H])$ and similarly define $R(A,B)$, $T(A,B)$ etc.

Let $P\le G$ and $Q\le H$ be subgroups and $\phi:Q\to P$ a group isomorphism.  The subgroup $\Delta(P,\phi,Q):=\{(\phi(q),q)\mid\, q\in Q\}\le G\times H$ is called {\em twisted diagonal}. We denote by $\TD(A,B)$ the Grothendieck group with respect to split short exact sequences of $p$-permutation $(A,B)$-bimodules whose indecomposable summands have twisted diagonal vertices when regarded as $\calO [G\times H]$-modules. 
\end{nothing}

\begin{nothing}
Let $\Delta(P,\phi,Q)\le G\times H$ be a $p$-subgroup.  Following the notation in \cite[10.1]{BoltjePerepelitsky2020}, for an element $\gamma\in T^\Delta(\calO G, \calO H)$, we write $\overline{\gamma}(P,\phi,Q)$ for the Brauer construction $\gamma(\Delta(P,\phi,Q))\in T(kN_{G\times H}(\Delta(P,\phi,Q))$.  Set $N:=N_{G\times H}(\Delta(P,\phi,Q)$. The corresponding elements in the commutative diagram (see \cite[9.1(c)]{BoltjePerepelitsky2020})

\begin{center}
\begin{tikzcd}
T(\calO N)\arrow{d}[swap]{\cong}\arrow{r}{\kappa_N}& R(\KK N)\ar{d}{d_N}\\
T(kN)\arrow{r}[swap]{\eta_N}& R(kN)
\end{tikzcd}
\end{center}

will be denoted by 

\begin{center}
\begin{tikzcd}
\gamma(P,\phi,Q)\arrow{d}\arrow{r}& \mu(P,\phi,Q)\ar{d}\\
\overline{\gamma}(P,\phi,Q)\arrow{r}& \nu(P,\phi,Q)
\end{tikzcd}
\end{center}
where $\kappa_N$ is induced by the scalar extension $\KK\otimes_\calO -$, $d_N$ is the decomposition map and $\eta_N$ is induced by the map $[M]\mapsto [M]$.
\end{nothing}

\begin{nothing}\label{noth Lambda}
Let $A$ be a block of $kG$ and $B$ a block of $kH$. Let $(P,e)$ be an $A$-Brauer pair. We denote by $\Lambda_H$ the set of pairs $(\phi, (Q,f))$ where $(Q,f)$ is a $kH$-Brauer pair and $\phi:Q\to P$ is an isomorphism. The group $N_G(P,e)\times H$ acts on $\Lambda_H$ via $(g,h)\cdot (\phi, (Q,f))=(c_g\phi c_h^{-1},\lexp{h}{(Q,f)})$. 

We also set $\Lambda_B\subseteq \Lambda_H$ to be the subset consisting of the pairs $(\phi, (Q,f))$ where $(Q,f)$ is a $B$-Brauer pair.  Note that $\Lambda_B$ is still an $N_G(P,e)\times H$-set via the above action. We denote by $\tilde{\Lambda}_H$ a set of representatives of the $H$-orbits of $\Lambda_H$ and set $\tilde{\Lambda}_B:=\tilde{\Lambda}_H\cap \Lambda_B$.
\end{nothing}

The crucial point in the proof of Theorem~\ref{thm main thm} is to observe that Lemma~10.3 in \cite{BoltjePerepelitsky2020} can be generalized as follows.

\begin{proposition}\label{prop perfectisometry}
Let $A\in \Bl(\calO G)$ be a block algebra and let $B=B_1\times\cdots\times B_m$ be a direct product of block algebras where $B_j\in\Bl(\calO H_j)$.  For each $j\in \{1,\cdots, m\}$, let $\gamma_j\in T^\Delta(A,B_j)$ be such that
\begin{align}\label{eqn equividentity}
\gamma_1\cdot_{H_1} \gamma_1^\circ + \cdots + \gamma_m\cdot_{H_m} \gamma_m^\circ = [A] \in T^\Delta(A,A)\,.
\end{align}
Let also $(P,e)$ be an $A$-Brauer pair. Consider the set of pairs $\Lambda_{B_j}\subseteq \Lambda_{H_j}$ as in \ref{noth Lambda}. Then there exists a unique $j\in\{1,\cdots,m\}$ and a unique $H_j$-orbit of pairs $(\phi_j,(Q_j, f_j))\in\Lambda_{B_j}$ such that
\begin{align*}
e\mu_j(P,\phi_j,Q_j)f_j \neq 0 \quad \text{in} \quad R(\KK C_G(P)e, \KK C_{H_j}(Q_j)f_j)\,.
\end{align*}
Moreover,  $e\mu_j(P,\phi_j,Q_j)f_j $ is a perfect isometry between $\KK C_G(P)e$ and $\KK C_{H_j}(Q_j)f_j$ and $e\nu_j(P,\phi,Q_j)f_j\neq 0$ in $R(kC_G(P)e, kC_{H_j}(Q_j)f_j)$.
\end{proposition}
\begin{proof}
The proof of this lemma is similar to the proof of \cite[Lemma~10.3]{BoltjePerepelitsky2020}.  The key point is to observe that Corollary 8.8 in \cite{BoltjePerepelitsky2020} is still applicable in this case.  We add a sketch of the proof for the convenience of the reader.

Apply the Brauer construction with respect to $\Delta(P)$ to Equation~(\ref{eqn equividentity}). The equality
\begin{align*}
\left[kC_G(P)e\right]&=\left[eA\left(\Delta(P)\right)e\right]=e\left(\sum_{j=1}^m \left(\gamma_j\cdot_{H_j}\gamma_j^\circ\right)\left(\Delta(P)\right)\right)e\\&
=\sum_{j=1}^m\sum_{\left(\phi_j,(Q_j,f_j)\right)\in\tilde{\Lambda}_{H_j}}e\overline{\gamma}_j\left(P,\phi_j,Q_j\right) f_j\cdot_{C_{H_j}(Q_j)}f_j\overline{\gamma^\circ}(Q_j,\phi_j^{-1},P)e
\end{align*}
holds in $T^\Delta\left(kC_G(P)e,kC_G(P)e\right)$. Lifting this equation from $k$ to $\calO$ and extending the scalars to $\KK$, we get
\begin{align*}
\left[\KK C_G(P)e\right]=\sum_{j=1}^m\sum_{\left(\phi_j,(Q_j,f_j)\right)\in\tilde{\Lambda}_{H_j}}e\mu_j\left(P,\phi_j,Q_j\right) f_j\cdot_{C_{H_j}(Q_j)}\left(e\mu_j\left(P,\phi_j,Q_j\right) f_j\right)^\circ
\end{align*}
in $R\left(\KK C_G(P)e,\KK C_G(P)e\right)$. The statement follows now from Corollaries 8.8 and 8.11 in \cite{BoltjePerepelitsky2020}.
\end{proof}

Now we can prove a weaker version of Theorem~\ref{thm main thm}.

\begin{corollary}\label{cor numberofproducts}
Let $A=A_1\times\cdots\times A_n$ and $B=B_1\times\cdots\times B_m$ be direct products of block algebras where $A_i\in\Bl(\calO G_i)$ and $B_j\in\Bl(\calO H_j)$ with $a_i$ and $b_j$ their respective identity elements. Assume that there exists a $p$-permutation equivalence $\gamma=(\gamma_{ij})$ between $A$ and $B$.  Then for each $i\in\{1,\cdots,n\}$ there exists a unique $j\in\{1,\cdots,m\}$ such that 
\begin{align*}
\mu_{ij}\neq 0 \quad \text{in} \quad R(\KK G_ia_i, \KK H_j b_j)\,.
\end{align*}
This defines a bijection between the sets $\{1,\cdots,n\}$ and $\{1,\cdots,m\}$.  In particular, we have $n=m$ and if $A_i$ and $B_j$ are corresponding blocks via the bijection above, then $\mu_{ij}$ is a perfect isometry between $\KK G_ia_i$ and $\KK H_jb_j$.  
\end{corollary}
\begin{proof}
Let $i\in\{1,\cdots,n\}$.  Since $\gamma$ is a $p$-permutation equivalence between $A$ and $B$, we have
\begin{align*}
\gamma_{i1}\cdot_{H_1} \gamma_{i1}^\circ + \cdots + \gamma_{im}\cdot_{H_m} \gamma_{im}^\circ = [A_i] \in T^\Delta(A_i,A_i)\,.
\end{align*}
Proposition \ref{prop perfectisometry} applied to the $A_i$-Brauer pair $(\{1\}, a_i)$ implies that there exists a unique $j\in\{1,\cdots,m\}$ such that 
\begin{align*}
\mu_{ij}\neq 0 \quad \text{in} \quad R(\KK G_ia_i, \KK H_j b_j)\,.
\end{align*}
Since by symmetry, a similar statement holds for every element $j\in\{1,\cdots,m\}$ it follows that $\gamma$ is a square matrix and in each row and in each column of $\gamma$ there exists a unique entry with a nonzero image in the corresponding character ring.  The last statement also follows from Proposition \ref{prop perfectisometry}. 
\end{proof}

The following is essentially Lemma~10.4 in \cite{BoltjePerepelitsky2020}. One can easily follow the proof of Lemma~10.4 in \cite{BoltjePerepelitsky2020} and make the necessary changes as we did in the proof of Proposition~\ref{prop perfectisometry} to prove it.

\begin{proposition}\label{prop techprop}
Let $A\in \Bl(\calO G)$ be a block algebra and let $B=B_1\times\cdots\times B_m$ be a direct product of block algebras where $B_j\in\Bl(\calO H_j)$.  For each $j\in \{1,\cdots, m\}$, let $\gamma_j\in T^\Delta(A,B_j)$ be such that
\begin{align}\label{Eqn rowequality}
\gamma_1\cdot_{H_1} \gamma_1^\circ + \cdots + \gamma_m\cdot_{H_m} \gamma_m^\circ = [A] \in T^\Delta(A,A)\,.
\end{align}
Let $(P,e)$ be an $A$-Brauer pair and set $I=N_G(P,e)$ and $X=N_{I\times I}\left(\Delta(P)\right)$.  For each $j\in\{1,\cdots, m\}$ consider the set $\Lambda_{B_j}$ together with its $I\times H_j$-action from \ref{noth Lambda}.  For $\lambda_j=\left(\phi_j,(Q_j, f_j)\right)\in\Lambda_{B_j}$ we set 
\begin{align*}
J(\lambda_j):=N_{H_j}(Q_j,f_j)\,,\quad I(\lambda_j):=N_{(I,\phi_j, J(\lambda_j))}\le I, \quad \text{and} \quad X(\lambda_j):=N_{I\times J(\lambda_j)}(\Delta(P,\phi_j,Q_j))\,.
\end{align*}
Then, $X* X(\lambda_j)=X(\lambda_j)$, and for each $\chi\in\mathrm{Irr}(\KK X(e\otimes e^*))$, there exists a unique $j\in \{1,\cdots, m\}$ and a unique $I\times H_j$-orbit of pairs $\lambda_j=\left(\phi_j,(Q_j, f_j)\right)\in\Lambda_{B_j}$ such that
\begin{align*}
\chi \cdot_G^{X,X(\lambda_j)} e\mu_j(P,\phi_j, Q_j)f_j \neq 0 \quad \text{in} \quad R(\KK [X(\lambda_j)](e\otimes f_j^*)\,.
\end{align*}
Moreover, for each $\lambda_j=\left(\phi_j,(Q_j, f_j)\right)\in\Lambda_{B_j}$ satisfying this condition, one has
\begin{align*}
\chi \cdot_G^{X,X(\lambda_j)} e\mu_j(P,\phi_j, Q_j)f_j \in  \pm \mathrm{Irr}(\KK [X(\lambda_j)](e\otimes f_j^*)\,.
\end{align*}
\end{proposition}

\begin{Remark}
Suppose that we have
\begin{align}
\gamma_1\cdot_{H_1} \gamma_1^\circ + \cdots + \gamma_m\cdot_{H_m} \gamma_m^\circ = [A] \in T^\Delta(A,A)
\end{align}
as in Proposition~\ref{prop perfectisometry}. Since by Proposition~\ref{prop techprop}, the results of Lemma~10.4 in \cite{BoltjePerepelitsky2020} hold, it follows that Corollaries 10.5 and 10.6 in \cite{BoltjePerepelitsky2020} are still valid in this case as well.
\end{Remark}

\begin{corollary}\label{cor permequiv}
Let $A\in \Bl(\calO G)$ be a block algebra with identity element $a$ and let $B=B_1\times\cdots\times B_m$ be a direct product of block algebras where $B_j\in\Bl(\calO H_j)$ with identity element $b_j$.  For each $j\in \{1,\cdots, m\}$, let $\gamma_j\in T^\Delta(A,B_j)$ be such that
\begin{align*}
\gamma_1\cdot_{H_1} \gamma_1^\circ + \cdots + \gamma_m\cdot_{H_m} \gamma_m^\circ = [A] \in T^\Delta(A,A)\,.
\end{align*}
Then there exists a unique $j\in\{1,\cdots, m\}$ such that $\gamma_j\neq 0$ in $T^\Delta(A,B_j)$.
\end{corollary}
\begin{proof}
By Corollary \ref{cor numberofproducts}, there exists a unique $j\in \{1,\cdots, m\}$ such that
\begin{align*}
\mu_j \neq 0 \quad \text{in} \quad R(\KK Ga, \KK H_j b_j)\,.
\end{align*}
This means that for any $j'\in\{1,\cdots, m\}$ with $j'\neq j$, one has $
\mu_{j'} = 0$ in $R(\KK Ga, \KK H_{j'} b_{j'})$. For every $A\otimes B_{j'}^*$-Brauer pair $\left(\Delta(P,\phi, Q), (e\otimes f^*)\right)$, since $\left(\{1\}, (a\otimes b_{j'})\right)\le \left(\Delta(P,\phi, Q), (e\otimes f^*)\right)$ holds, \cite[Corollary~10.6]{BoltjePerepelitsky2020} implies that
\begin{align*}
e\mu_{j'}(P,\phi,Q)f = 0 \quad \text{in} \quad R(\KK [C_G(P)]e, \KK[C_{H_j}(Q)]f)\,.
\end{align*}
Therefore, by \cite[Corollary~10.5]{BoltjePerepelitsky2020} one has
\begin{align*}
e\mu_{j'}(P,\phi,Q)f = 0 \quad \text{in} \quad R(\KK [N_{G\times H_j}(\Delta(P,\phi_j,Q_j))](e\otimes f^*)\,.
\end{align*}

This shows that the element $\gamma_{j'}$ is in the kernel of the injective map in \cite[Proposition~9.2(b)]{BoltjePerepelitsky2020} and hence equals to zero. 
\end{proof}

{\em Proof of Theorem~\ref{thm main thm}}: The fact that $n=m$ follows from Corollary~\ref{cor numberofproducts}. For each $i\in\{1,\cdots,n\}$, by Corollary~\ref{cor permequiv}, there exists a unique $j\in\{1,\cdots,m\}$ such that $\gamma_{ij}\neq 0$ in $\TD(A_i,B_j)$. This proves the theorem. \qed

\section*{Acknowledgement}
I would like to thank Robert Boltje for his suggestion to study $p$-permutation equivalences between direct product of blocks. 


\centerline{\rule{5ex}{.1ex}}
\begin{flushleft}
Deniz Y\i lmaz, Department of Mathematics, Bilkent University, 06800 Ankara, Turkey.\\
{\tt d.yilmaz@bilkent.edu.tr}
\end{flushleft}


\begin{thebibliography}{XXXX}



\bibitem[BoP20]{BoltjePerepelitsky2020}{\sc R.~Boltje, P.~Perepelitsky:}
        $p$-permutation equialences between blocks of group algebras.
        arXiv:2007.09253.
        
\bibitem[BoX08]{BoltjeXu2008} {\sc R.~Boltje, B.~Xu:}
	On $p$-permutation equivalences: Between Rickard equivalences and isotypies.
	{\sl Transactions of the American Mathematical Society} {\bf 360} (2008), 5067--5087.

\bibitem[BY22]{BoucYilmaz2022}{\sc S.~Bouc, D.~Y\i lmaz:}
	Diagonal $p$-permutation functors, semisimplicity, and functorial equivalence of blocks.
	{\sl Advances in Mathematics} {\bf 411} (2022), 108799.

	         

\bibitem[Br90]{Broue1990}{\sc M.~Brou{\'e}:}
	Isom{\'e}tries parfaites, types de blocs, cat{\'e}gories d{\'e}riv{\'e}es. 
	{\sl Ast{\'e}risque} No. {\bf 181-182} (1990), 61--92.  
	
\bibitem[R96]{Rickard1996} {\sc J.~Rickard:}
	Splendid equivalences: derived categories and permutation modules.
	{\sl Proceedings of the London Mathematical Society} {\bf 72} (1996), 331--358.
	
\bibitem[P99]{Puig1999} {\sc L.~Puig:}
	On the local structure of Morita and Rickard equivalences between Brauer blocks.
	Progress in Mathematics, 178. Birkh{\"a}user Verlag, Basel, 1999.
	

\end{thebibliography}
\end{document}